\documentclass[12pt]{article}

\usepackage{geometry} 
\usepackage{amsthm,amssymb}
\usepackage{amsmath,amsfonts,latexsym}
\usepackage{latexsym}
\usepackage{float}
\usepackage{xcolor}
\usepackage{url}

\newtheorem{thm}{Theorem}[section]

\newtheorem{cor}[thm]{Corollary}
\theoremstyle{definition}
\newtheorem{rem}[thm]{Remark}
\newtheorem{defi}[thm]{Definition}
\newtheorem{example}[thm]{Example}

\begin{document}
\vspace*{0.5cm}

\begin{center}
{\Large\bf Transitive regular $q$-analogs of graphs}
\end{center}

\vspace*{0.5cm}

\begin{center}
        Dean Crnkovi\'c  (deanc@math.uniri.hr)\\
        Vedrana Mikuli\' c Crnkovi\' c (vmikulic@math.uniri.hr)\\				
		Andrea \v Svob (asvob@math.uniri.hr)\\
		and \\
		Matea Zubovi\' c \v Zutolija (matea.zubovic@math.uniri.hr)\\[3pt]
		{\it\small Faculty of Mathematics, University of Rijeka} \\
		{\it\small Radmile Matej\v ci\'c 2, Rijeka, Croatia}\\
\end{center}

\vspace*{0.2cm}


\vspace*{0.5cm}

\begin{abstract}
In 1976, Delsarte introduced the notion of $q$-analogs of designs, and $q$-analogs of graphs were introduced recently by M. Braun et al.
In this paper, we extend that study by giving a method for constructing transitive regular $q$-analogs of graphs. Further, we illustrate the method by giving some examples.
Additionally, we introduced the notion of $q$-analogs of quasi-strongly regular graphs and give examples of transitive $q$-analogs of quasi-strongly regular graphs coming from spreads. 
\end{abstract}

\bigskip



\section{Introduction}

In \cite{delsarte}, Delsarte introduced the notion of $q$-analogs of designs. Let $\mathbb{F}_q^v$ be a vector space of dimension $v$ over a finite field $\mathbb{F}_q$. A subspace
of vector space $\mathbb{F}_q^v$ of dimension $k$ will be called a $k$-subspace. A $t$-$(v,k,\lambda;q)$ subspace design
consists of a collection $\mathcal{B}$ of $k$-subspaces of $\mathbb{F}_q^v$, called blocks, such that each $t$-subspace of $\mathbb{F}_q^v$ is contained in exactly $\lambda$ blocks. 

A vector space partition of $\mathbb{F}_q^v$ is a set $S$ of $t$-subspaces, $t \ge 2$, such that each 1-dimensional subspace of $\mathbb{F}_q^v$ is contained in precisely one element of $S$. 
Let $S$ be a vector space partition consisting of $t$-subspaces and let $\mathcal{B}$ be a collection of $k$-dimensional subspace of $\mathbb{F}_q^v$. A triple $(\mathbb{F}_q^v,S,\mathcal{B})$ is a 
$(v,t,k, \lambda)$ group divisible design over $\mathbb{F}_q$, or a $(v,t,k, \lambda;q)$ group divisible design, if any 2-dimensional subspace of $\mathbb{F}_q^v$ is either contained in exactly one member of $S$ or contained in exactly $\lambda$ members of $\mathcal{B}$, but not both. The elements of $\mathcal{B}$ are called blocks.

A 2-design or a group divisible design over a field is said to be simple if it does not have repeated blocks.

The group of all linear bijective mappings on $\mathbb{F}_q^v$ is denoted by $GL(v,q)$, and the group of all semilinear bijective mappings on $\mathbb{F}_q^v$ is denoted by $\Gamma L(v,q)$. It holds that $\Gamma L(v,q) \cong GL(v,q) : Aut(\mathbb{F}_q)$, a semidirect product of $GL(v,q)$ and $Aut(\mathbb{F}_q)$, where $Aut(\mathbb{F}_q)$ is the full automorphism group of the field 
$\mathbb{F}_q$. The groups $GL(v,q)$ and $\Gamma L(v,q)$ act faithfully on the vectors, but not on the subspaces of $\mathbb{F}_q^v$. However, the groups $PGL(v,q)=GL(v,q)/C$ and 
$P \Gamma L(v,q)= \Gamma L(v,q) / C$, where $C= \{ x \mapsto \alpha x : \alpha \in \mathbb{F}_q \}$, act faithfully on the subspaces. 
It holds that $P \Gamma L(v,q) \cong PGL(v,q) : Aut(\mathbb{F}_q)$. 

Let $\mathcal{D}=(V,\mathcal{B})$ be a subspace design which consists of a set $\mathcal{B}$ of $k$-subspaces of $V=\mathbb{F}_q^v$.
An automorphism of $\mathcal{D}$ is a bijective mapping $g \in P\Gamma L(v,q)$ which maps blocks to blocks.
The set of all automorphisms of a subspace design $\mathcal{D}$ forms a group under the composition of mappings, called the full automorphism group of $\mathcal{D}$ and denoted by $Aut(\mathcal{D})$ (see \cite{computational}). Any subgroup of $Aut(\mathcal{D})$ is called an automorphism group of $\mathcal{D}$.

In 1987, Thomas (see \cite{thomas}) constructed the first non-trivial $q$-analogs of design with parameters 
$2$-$(n,3,7;2),\ n>6, n=6k+1$ or $n=6k-1$. Further, Suzuki constructed $2$-$(n,3,q^2+q+1;q)$ design, $n>6, n=6k+1$ or $n=6k-1$ (see \cite{suz1,suz2}), 
and in \cite{mune}, Miyakawa, Munemasa, and Yoshiara gave a classification of $2$-designs for which $n\in \{6,7\},\ q\in \{2,3\},\ k=3$, 
having an automorphism group that acts transitively on the 1-subspaces.

Due to applications of $q$-analogs of designs to network coding and distributed storage, the interest in this subject has been renewed recently (see \cite{mario-subspace}). 
One of the most important recent results was given in \cite{qsteiner}, where the authors constructed designs over a finite field with parameters $2$-$(13,3,1;2)$ which were the first known examples of Steiner $q$-designs that do not arise from spreads. Further, in \cite{buratti} the authors apply difference methods to prove that there exists a cyclic 2-$(n,3,7;2)$ design for every odd positive integer $n$, which improves the result of Thomas given in \cite{thomas}. In \cite{buratti}, it is also shown that there exists a cyclic and simple $(n,3,3,7;2)$ group divisible design for every integer 
$n \equiv 3 \ (mod \ 6)$, which is the first construction of an infinite family of cyclic group divisible designs over a finite field. 

Recently, $q$-analogs of graphs were introduced and $q$-ary strongly regular graphs were studied (see \cite{q-srg}). 
A $q$-ary graph $\mathcal{E}$ in is a pair $(V,E)$, where $V={\mathbb{F}_q^v \brack 1}$ and $E \subseteq {\mathbb{F}_q^v \brack 2}$, where ${\mathbb{F}_q^v \brack k}$ is the set of all 
$k$-dimensional subspaces of $\mathbb{F}_q^v$. The elements of $V$ and $E$ are called vertices and edges, respectively. If there is an edge $x \oplus y$, where $x$ and $y$ are 1-subspaces of $\mathbb{F}_q^v$, then we say that $x$ and $y$ are adjacent. For a vertex $x$ of a $q$-ary graph ${\mathcal E}$ we define its neighborhood $N_{{\mathcal E}}(x)$ as the set of all 1-subspaces adjacent to it, and itself.
A $q$-ary graph ${\mathcal E}$ is called $k$-regular if for any vertex $x$ the neighborhood $N_{{\mathcal E}}(x)$ is a $(k+1)$-dimensional subspace. As shown in \cite[Theorem 3.16.]{q-srg}, a $k$-regular
$q$-ary graph in $\mathbb{F}_q^v$ yields a classical $([k+1]_q-1)$-regular graph on $[v]_q$ vertices, where $[n]_q=\frac{q^n-1}{q-1}$.
Further, the set of all neighborhoods of a $k$-regular $q$-ary graph in $\mathbb{F}_q^v$ form a $q$-ary design with parameters 1-$(v,k+1,[k+1]_q;q)$.
It is a so-called neighborhood design of a regular $q$-ary graph. 

Let $\mathcal{E}=(V,E)$ be a $q$-ary graph in $\mathbb{F}_q^v$. An automorphism of $\mathcal{E}$ is a bijective mapping $g \in P\Gamma L(v,q)$ which maps edges to edges.
The set of all automorphisms of $\mathcal{E}$ forms a group called the full automorphism group of $\mathcal{E}$, denoted by $Aut(\mathcal{E})$.
Any subgroup of $Aut(\mathcal{E})$ is called an automorphism group of $\mathcal{E}$.
Note that every automorphism of a $q$-ary graph yields an automorphism of the corresponding classical graph. The converse is not true, since an automorphism of a $q$-ary graph is determined by its action on a basis of the vector space, while automorphisms of classical graphs do not have this constraint. A similar statement holds for $q$-ary designs and corresponding classical designs.

In this paper, we present a method of constructing regular $q$-ary graphs that have automorphism groups acting transitively on the 1-subspaces.
Further, we introduced the notion of $q$-analogs of quasi-strongly regular graphs and give examples of transitive $q$-analogs of quasi-strongly regular graphs coming from spreads.

\section{Constructing transitive $q$-ary graphs}

A method of constructing primitive designs and graphs is given in \cite{km}. The following generalization of the method by Key and Moori is given in \cite{cms}.

\begin{thm} \label{t-des}
Let $G$ be a finite permutation group acting transitively on the sets $\Omega_1$ and $\Omega_2$
of size $m$ and $n$, respectively.
Let $\alpha \in \Omega_1$ and $\Delta_2 =  \bigcup_{i=1}^s \delta_i G_{\alpha}$, where $G_{\alpha} = \{ g \in G \ | \ \alpha g = \alpha \}$ is the stabilizer of $\alpha$
and $\delta_1,...,\delta_s \in \Omega_2$ are representatives of distinct $G_\alpha$-orbits on $\Omega_2$.
If $\Delta_2 \neq \Omega_2$ and
$${\mathcal{B}}=\{ \Delta_2 g : g \in G \},$$
then ${\mathcal{D}}(G,\alpha,\delta_1,...,\delta_s)=(\Omega_2,{\mathcal{B}})$ is a
$1$-$(n, | \Delta_2 |, \frac{|G_{\alpha}|}{|G_{\Delta_2}|}\sum_{i=1}^{s} | \alpha G_{\delta_i} |)$ design with $\frac{m\cdot |G_{\alpha}|}{|G_{\Delta_2}|}$ blocks.
The group $H\cong G/{\bigcap_{x\in \Omega_2}G_x}$ acts as an automorphism group on $(\Omega_2,{\mathcal{B}})$, 
transitively on points and blocks of the design.
\end{thm}

If $\Delta_2=\Omega_2$, then the set $\mathcal{B}$ consists of one block, and  ${\mathcal{D}}(G,\alpha,\delta_1,...,\delta_s)$ is a design with parameters $1$-$(n,n,1)$.
Suppose that the group $G$ acts transitively on the points and blocks of a simple $t$-$(v,k,\lambda)$ design. Then the design can be obtained as described in Theorem \ref{t-des}.
If $\Omega_1=\Omega_2$ and $\Delta_2$ is a union of self-paired and mutually paired orbits of $G_{\alpha}$, then the design 
${\mathcal{D}}(G,\alpha,\delta_1,...,\delta_s)$ is a symmetric self-dual design and the incidence matrix of that design is the 
adjacency matrix of a $|\Delta_2|-$regular graph. 

In \cite{construction}, we present a method for constructing $q$-analogs of 1-designs having an automorphism group that acts transitively both on the $1$-subspaces of the vector space $\mathbb{F}_q^v$ and on the set of blocks of the subspace designs. The same construction is given in Theorem \ref{main}.

\begin{thm}\label{main}
Let $G< P \Gamma L(v,q)$ be a group acting transitively on the set of $1$-subspaces of the vector space $\mathbb{F}_q^v$. 
Let $P$ be a subgroup of the group $G$ and $\Delta=\bigcup_{i=1}^s \delta_i P$, for $1$-spaces $\delta_1,..., \delta_s$ which are representatives of different $P$-orbits. 
If $\Delta$ is a subspace of $\mathbb{F}_q^v$, then $\mathcal{B}=\{\Delta g\ |\ g\in G\}$ is a set of blocks of a $q$-analog of design with parameters
$1$-$\left (v,k,\frac{|G|(q^k-1)}{|G_{\Delta}|(q^v-1)};q\right )$ and $\frac{|G|}{|G_{\Delta}|}$ blocks, where $k=\dim \Delta$.
The group $G$ acts as an automorphism group on the constructed $q$-analog of design, transitively on the $1$-subspaces of the vector space $\mathbb{F}_q^v$ and on the set of blocks of the design.
\end{thm}

Note that $P \le G_{\Delta}$, where $P$ and $\Delta$ are defined as in Theorem \ref{main}.
Since $\Delta$ must be a subspace of $\mathbb{F}_q^v$, it can be constructed as $\Delta=\Sigma_{i=1}^s \langle \delta_iP \rangle$.
If the group $G$ acts transitively on the points and blocks of a simple design $t$-$(v,k,\lambda ; q)$, then the design can be obtained as described in Theorem \ref{main} (see \cite{construction}).

\begin{rem}
It is shown in \cite{dan-q-designs} that if $t \ge 2$ and $\mathcal{D}$ is a block-transitive $t$-$(v,k, \lambda ;q)$ design then $\mathcal{D}$ is trivial, that is, the set of blocks of $\mathcal{D}$ is the set of all $k$-dimensional subspaces of $\mathbb{F}_q^v$.
However, block-transitive $q$-ary 1-designs can be used as "building blocks" of $q$-ary $t$-designs, for $t \ge 2$. For example, the $q$-ary designs with parameters 2-(13,3,1;2) constructed in \cite{qsteiner} admit a transitive action of the normalizer $G$ of the Singer subgroup in $GL(13,2)$ on 1-subspaces, and $G$ acts in 15 orbits on blocks. Hence, each of these $q$-ary 2-designs consists of 15 block-transitive $q$-ary 1-designs, and each one of these transitive $q$-ary 1-designs can be constructed using Theorem \ref{main}.
\end{rem}

The following theorem gives us the method of constructing transitive $q$-ary graphs. 

\begin{thm} \label{thm-q-graphs}
Let $G<P \Gamma L(v,q)$ act transitively on the set of $1$-subspaces of the vector space $\mathbb{F}_q^v$. Further, let $P$ be the stabilizer of a $1$-dimensional subspace $\alpha$ for that action and 
$\Delta=\bigcup_{i=1}^s \delta_i P$, where $\delta_1= \alpha$ and $\delta_1, \ldots ,\delta_s$ are representatives of different $P$-orbits. Suppose that $\Delta$ is a subspace of $\mathbb{F}_q^v$.
Let $\mathcal{B}=\{\Delta g\ |\ g\in G\}$ be the set of blocks of the $q$-analog of $1$-design with the base block $\Delta$. 
If $\{ \alpha g : g \in G, \ \alpha g^{-1} \in \Delta \} = \Delta$, then $\mathcal{B}$ is a set of neighborhoods of a $(\dim \Delta -1)$-regular $q$-graph ${\mathcal E}$ and 
$\mathcal{B}'= \{ (\alpha g,\Delta g) \ |\ g\in G \}$ is the set of ordered pairs $(x, N_{{\mathcal E}}(x))$,  where $x \in V={\mathbb{F}_q^v \brack 1}$.
The group $G$ acts as an automorphism group on ${\mathcal E}$, transitively on the $1$-subspaces of the vector space $\mathbb{F}_q^v$ and on the set of neighborhoods of ${\mathcal E}$.
\end{thm}
\begin{proof}
The subspace $\Delta$ is the neighborhood of $\alpha$, and $\Delta g$, $g \in G$, is the neighborhood of the 1-subspace $\alpha g$. The condition $\{ \alpha g : g \in G, \ \alpha g^{-1} \in \Delta \} = \Delta$ implies that $x \in N_{{\mathcal E}}(y)$ if and only if $y \in N_{{\mathcal E}}(x)$, for all 1-subspaces $x$ and $y$ of $\mathbb{F}_q^v$. Hence, $\mathcal{B}$ is the set of neighborhoods of a $(\dim \Delta -1)$-regular $q$-graph $\mathcal{E}$, and $G$ acts as an automorphism group of $\mathcal{E}$.
\end{proof}

Every transitive regular $q$-ary graph can be constructed as described in Theorem \ref{thm-q-graphs}.

\begin{cor} \label{cor-q-graphs}
Let ${\mathcal E}$ be a regular $q$-ary graph in $\mathbb{F}_q^v$ having an automorphism group $G$ that acts transitively on the 1-subspaces. 
Then ${\mathcal E}$ can be constructed by the method given in Theorem \ref{thm-q-graphs}.
\end{cor}
\begin{proof}
Let $\alpha$ be a 1-subspace of $\mathbb{F}_q^v$ and let $P$ be the stabilizer of $\alpha$ in $G$. Further, let $\Delta$ be the neighborhood of $\alpha$. Then $\{ \alpha g : g \in G, \ \alpha g^{-1} \in \Delta \} = \Delta$. Further, $\Delta$ is a union of $P$-orbits and $\Delta$ can be obtained as given in Theorem \ref{thm-q-graphs}. The neighborhoods of other 1-subspaces of $\mathbb{F}_q^v$ are $G$-images of $\Delta$, such that $\Delta g$ is the neighborhood of $\alpha g$, for $g \in G$.
\end{proof}

The strongly regular $q$-ary graphs given in Examples \ref{ex-symplectic} and \ref{ex-spread} are the only regular $q$-ary graphs that admit a transitive action of some 
$P \Gamma L(v,q)$, for $(v,q) \in \{(3,2), (3,3), (3,4), (4,2)\}$. The constructions are obtained by using GAP \cite{GAP2022} and the GAP package FinInG \cite{fining}.

\begin{example} \label{ex-symplectic}
The strongly regular $q$-ary graphs with parameters $(v, v-2, [v-2]_q -2, [v-2]_q; q)$, with $v$ even, given in \cite[Example 3.13]{q-srg}, admit actions of the symplectic groups $PSp(v,q)$, transitive on 1-subspaces. Hence, all the $q$-ary graphs from this family of strongly regular $q$-ary graphs can be obtained using the construction described in Theorem \ref{thm-q-graphs}.
\end{example}

\begin{example} \label{ex-spread}
In \cite[Example 3.12]{q-srg}, regular $q$-ary graphs arising from vector space partitions, i.e. spreads, are given. 
Let $S$ be a vector space partition of $\mathbb{F}_q^v$ consisting of $t$-subspaces, $t \ge 2$.
Such a vector space partition corresponds to a $(t -1)$-spread of $PG(v-1, q)$. 
The set of 2-subspaces contained in an element of $S$ forms a $q$-ary strongly regular graph ${\mathcal E}$ with parameters $(v, t-1, [t]_q-2, 0; q)$, that is a $q$-analog of a classical strongly regular graph on $v$ vertices which is a disjoint union of complete graphs on $t$ vertices. 
By applying Theorem \ref{thm-q-graphs} we get the neighborhoods of ${\mathcal E}$, and the set of edges is $\{x \oplus y : x \in {\mathbb{F}_q^v \brack 1},\ y \in N_{{\mathcal E}}(x) \setminus \{ x \} \}$.
The information on a construction of some small examples of $q$-ary strongly regular graphs $(v, t-1, [t]_q-2, 0; q)$ is given in Table \ref{table-disjoint}.
If $S$ is a Desarguesian vector space partition (see \cite{des-spreads}) and $G$ is a group stabilizing the partition setwise, then $G$ can be described as an extension of a group acting on the 1-subspaces of $\mathbb{F}^{v/t}_{q^t}$ and a group acting on the 1-subspaces of $\mathbb{F}^t_q$ (see \cite{gvdv}). The groups in Table \ref{table-disjoint} are written in that form. 

\begin{table}[H]
\centering
\begin{tabular}{|c|c|c|c|c|}
\hline
$v$&$q$& $G < P \Gamma L(v,q)$ & Stabilizer $P$ of a 1-subspace & Parameters\\
\hline
4 & 2 & $\mathbb{Z}_{3} : (\mathbb{Z}_{5} : \mathbb{Z}_4)$ &
$\mathbb{Z}_4$ & $(4,1,1,0;2)$\\
\hline
4 & 3 & $ \mathbb{Z}_4.((A_6 : \mathbb{Z}_2): \mathbb{Z}_2)$ &  $ (E_9
: \mathbb{Z}_8) :\mathbb{Z}_2$ & $(4,1,2,0;3)$\\
\hline
6 & 2 &  $\mathbb{Z}_3. ((PSL(3,4) : \mathbb{Z}_3):\mathbb{Z}_2)$ &
$E_{16} : (A_5 : S_3)$ & $(6,1,1,0;2)$\\
 & &   $\mathbb{Z}_7 : (PSL(2,8) : \mathbb{Z}_3 )$ & $E_8 :
(\mathbb{Z}_7 : \mathbb{Z}_3)$ & $(6,2,5,0;2)$\\
\hline
9 & 2 & $\mathbb{Z}_7 : (PSL(3,8) : \mathbb{Z}_3)$ & $E_{64} :
(PSL(2,8) : (\mathbb{Z}_7 : \mathbb{Z}_3))$ & $(9,2,5,0;2)$\\
\hline
\end{tabular}
\caption{Examples of strongly regular $q$-ary graphs from spreads}
\label{table-disjoint}
\end{table}

\end{example}


\section{Quasi-strongly regular $q$-ary graphs}

Quasi-strongly regular graphs were introduced in \cite{QSRGs}, as a generalization of strongly regular graphs. A quasi-strongly regular graph of grade $n$ with parameters $(v,k,a,c_1, \ldots , c_n)$
is a $k$-regular graph with $v$ vertices such that any two adjacent vertices share $a$ common neighbours and any two non-adjacent vertices share $c_i$ common neighbours for some $1 \le i \le n$. 
The most interesting quasi-strongly regular graphs are those of grade 2. 

In the following definition, we introduce the notion of quasi-strongly regular $q$-ary graphs. 

\begin{defi}
A $k$-regular $q$-ary graph $\mathcal{E}$ in $\mathbb{F}_q^v$ is called a quasi-strongly regular $q$-ary graph of grade $n$ with parameters $(v,k,a,c_1, \ldots , c_n;q)$ if for all vertices $x,y$, 
with $x \ne y$, we have that
	\[
	|(N_{\mathcal{E}}(x)\cap N_{\mathcal{E}}(y))\setminus\{ x,y \}|=
	\begin{cases}
	  a                               & \text{if $x$ and $y$ are adjacent,}\\
	  c_i \in \{ c_1, \ldots , c_n \} & \text{otherwise.}
	\end{cases}
	\]
\end{defi}  

As in the case of classical quasi-strongly regular graphs, the most interesting among quasi-strongly regular $q$-ary graphs are those of grade 2. Below we give a construction of transitive quasi-strongly regular $q$-ary graphs of grade 2 coming from spreads.  

\begin{example} \label{quasi}
Let $S$ be a partition of the vector space $\mathbb{F}_q^v$ into $n$ subspaces of dimension $t$. Further, let $\mathcal{E}_1, \ldots, \mathcal{E}_n$ be strongly regular $q$-ary graphs with parameters 
$(t,k,\lambda,\mu ; q)$, such that $\mathcal{E}_i$ is defined on the $i$th element of the partition $S$. 
In that way, we obtain a disjoint union of the $q$-ary graphs $\mathcal{E}_1, \ldots, \mathcal{E}_n$, which is a quasi-strongly regular $q$-ary graph of grade 2 with parameters $(v,k,\lambda,\mu,0;q)$. 
Let the $q$-ary graphs $\mathcal{E}_1, \ldots, \mathcal{E}_n$ be transitive and let all of them be isomorphic to some $\mathcal{E}$. Then the obtained quasi-strongly regular $q$-ary graph, admits a transitive action of the group $Aut(\mathcal{E}). P \Gamma L(\frac{v}{t},q^t)$.
The construction from Example \ref{ex-spread} is a special case of the construction described here, obtained when $\mathcal{E}$ is a complete $q$-ary graph.
\end{example}  

It would be of interest to investigate properties of quasi-strongly regular $q$-ary graphs and find new constructions. Especially, it would be interesting to find constructions of quasi-strongly regular $q$-ary graphs of grade 2 that are not related to spreads. 

\vspace*{0.8cm}

\noindent {\bf Acknowledgement} \\
This work has been fully supported by {\rm C}roatian Science Foundation under the projects 4571 and 5713. 



\end{document}